\documentclass[a4paper,12pt]{article}
\usepackage{xcolor}
\usepackage{amsmath}
\usepackage{amssymb}
\usepackage{latexsym}
\usepackage{subcaption}
\usepackage{graphicx}
\usepackage{amsthm}

\newtheorem{theorem}{Theorem}[section]
\newtheorem{lemma}[theorem]{Lemma}

\newtheorem{Conjecture}[theorem]{Conjecture}

\theoremstyle{definition}
\newtheorem{definition}[theorem]{Definition}
\newtheorem{example}[theorem]{Example}
\newtheorem{remark}[theorem]{Remark}
\numberwithin{equation}{section}

\title{\LARGE{
Regularity of solutions for singular fractional differential equation}}

\begin{document}
%\maketitle

%%%%%%%%%%%%%%%%%%%%%%%%%%%%%%%%% Authors %%%%%%%%%%%%%%%%%%%%%%%%%%%%%%%%%%
\author{%
  Jinsil Lee${}^{a}$ and Yong-Hoon Lee${}^{b,1}$   \
       \\[1ex]
     {\small\itshape ${}^{a}$ Department of Mathematics, University of Georgia,} \\  
     {\small\itshape Athens, GA 30602, USA}\\ 
    {\small\itshape ${}^{b}$ Department of Mathematics, Pusan National University,} \\
    {\small\itshape Busan 46241, Republic of Korea}\\
     {\small\upshape E-mail: jl74942@uga.edu}\\
    {\small\upshape E-mail: yhlee@pusan.ac.kr}
    }

\footnotetext[1]{Corresponding Author}
%%%%%%%%%%%%%%%%%%%%%%%%%%%%%%%%%%%%%%%%%%%%%%%%%%%%%%%%%%%%%%%%%%%%%%%%%%%

\date{}
\maketitle

\begin{abstract}
In this work, we study the regularity of positive solutions for nonlinear fractional differential equation with a singular weight. We define the new Banach space and use this space to show the regularity. We also give an example with a singular weight which may not be in $L^1.$



\end{abstract}
%%%%%%%%%%%%%%%%%%%%%%%%%%%%%%%%%%%%%%%%%%%%%%%%%%%%%%%%%%%%%%%%%%%%%%%
\section{Introduction}
Many researchers study problems in physics, control theory, and chemistry, which can be represented as fractional differential equations (\cite{A, B}).
We are interested in the regularity of solutions from the following equations with a singular weight:
\begin{equation*}\tag*{$(FDE)$}\label{FDE}
%\quad\quad\quad\quad\quad\quad\quad
\begin{cases}
 D^{\alpha}_{0+}u(t)+h(t)f(u(t))= 0,\quad t\in (0,1),\\
u(0)= 0 = u(1),
\end{cases}
%\quad\quad\quad\quad\quad{(P)}
\end{equation*}
where $D^{\alpha}_{0+}$ is the Riemann-Liouville fractional derivative of order $\alpha \in (1,2]$, 
$f\in C([0,\infty),[0,\infty)) $ is a given continuous function and $h\in C((0,1],[0,\infty))$ satisfies the following conditions: 
\begin{flushleft}
$(H) \ \ \int_0^1 s^{\alpha -1}h(s)ds<\infty$,
\end{flushleft}
We notice that coefficient function $h$ satisfying condition $(H)$ may not be integrable near $t=0$, as an example, we may consider $h(t)=t^{-\beta}$ where $1<\beta <\alpha$. We see that $h$ satisfies conditions $(H)$ but $h\notin L^{1}((0,1),[0,\infty))$.
% \textcolor{red}{When we use a fixed point theorem, the Green's function plays an important role to set up corresponding integral operator for problem $(P)$.  }
  
Introducing the Green's function for the case that $h$ is continuous,
Bai and L\"u \cite{E} consider the following nonlinear problem 
%in general set-up by using a solution operator with Green's function
\begin{align}\label{eq:::1.1}
\begin{cases}
 D^{\alpha}_{0+}u(t)+f(t,u(t))= 0,\quad t\in (0,1),\\
u(0)= 0 = u(1),
\end{cases}
\end{align}
 where 
 %$D^{\alpha}_{0+}$ is the Riemann-Liouville fractional derivative of order $\alpha$,$1<\alpha\leq2$ is a real number, 
$f\in C([0,1]\times[0,\infty), [0,\infty)$). By taking the Riemann-Liouville fractional integral, they set up an equivalent solution operator $S$ by
\begin{align}
Su(t)= \int_{0}^{1} G(t,s)f(s,u(s))ds
\end{align}
where $G(t,s)$ defined by
\begin{align}\label{eq:::1.2}
G(t,s)=
\begin{cases}\displaystyle
 \frac{(t(1-s))^{\alpha-1}-(t-s)^{\alpha-1}}{\Gamma(\alpha)}, \quad 0\leq s\leq t\leq 1,\\\displaystyle
 \frac{(t(1-s))^{\alpha-1}}{\Gamma(\alpha)}, \quad 0\leq t\leq s\leq 1
\end{cases}
\end{align}
is the Green's function for the fractional differential equation
$$ D^{\alpha}_{0+}u(t)= 0$$
with Dirichlet boundary condition. Analysing this operator, they proved the existence of at least three positive solution of problem \eqref{eq:::1.1} in $C[0,1]$ under some additional conditions on $f$. 

However, if $h$ is not integrable, we should consider the existence of $D^\alpha_{0+}u$ and its solution space. Moreover, corresponding Green's function can not be obtained by obvious modification from the case $h\in L^1 .$ In the paper \cite{N}, the researchers considered the existence of the solution for the second order differential equation where the function $f$ is a given function satisfying Caratheodory's conditions with singularities at 0 and 1. They introduce the new Banach space $X=\{x\in C^1(0,1)| x\in C[0,1], \lim_{t\rightarrow 1}(1-t)x'(t) ~\text{and}~ \lim_{t\rightarrow 0}tx'(t) ~\text{exist}\}$ equipped with the norm $$\|x\|_X=\max_{t\in[0,1]}|x(t)|+\max_{t\in[0,1]}|t(1-t)x'(t)|$$ We are interested in extending the existence results to the fractional case. 
%$f(t,u(t))=t^{-\beta}u^{3}$ where $1<\beta <\alpha$, we cannot get the existence results using usual method.} 
%Motivated by this situation, 
In our paper, we define the solution space $E_\alpha$ and define the solution of our equation using this solution space and derive the Green's function in this singular situation which is one of our main goals for this paper. In our case, the solution $u$ may not be in $AC^2[0,1]$ so that we understand a solution $u$ is in $E_\alpha \cap AC[0,1] $ with $D^{\alpha-1}_{0+} u(t)$ which is absolutely continuous in any compact subinterval of (0,1) and $u$ satisfies the equation $(FDE_1)$ for $t\in [0,1]$ and boundary conditions.

The rest of the paper is organized as follows; In Section 2, we introduce some definitions and lemmas related to fractional calculus and Krasnoselski's classical fixed point theorem. Moreover, we introduce new Banach spaces as our solution spaces.
%some important theorems and lemmas which we will use later.
In Section 3, we derive the Green's function related to the problem with a singular weight and define a solution of our problem $\ref{FDE}$.

%%%%%%%%%%%%%%%%%%%%%%%%%%%%%%%%%%%%%%%%%%%%%%%%%%%%%%%%%%%%%%%%%%%%%%%%%%%%%%%%%%%%%%%%%%%%%
\section{Preliminaries}

In this section, we introduce some definitions of fractional calculus and some important lemmas, and a theorem that will be used later.
\begin{definition}{\upshape(\cite{Q}) }\label{def2}
The integral 
$$I^\alpha_{0+}u(t)=\frac{1}{\Gamma (\alpha)}\int_0^t \frac{u(s)}{(t-s)^{1-\alpha}}ds, ~t>0$$
where $\alpha >0$ is called the Riemann-Liouville fractional integral of order $\alpha$.
\end{definition}
\begin{definition}{\upshape(\cite{Q})}
For a function $u(t)$ given in the interval $[0,\infty )$, the expression 
$$D^\alpha_{0+}u(t)=\frac{1}{\Gamma(n-\alpha)}\Big{(}\frac{d}{dt}\Big{)}^n \int_0^t \frac{u(s)}{(t-s)^{\alpha-n+1}}ds$$
where $n=[\alpha]+1, [\alpha]$ denotes the integer part of number $\alpha,$ is called the Riemann-Liouville fractional derivative of order $\alpha$.
\end{definition}
\begin{remark}{\upshape(\cite{Q})} \upshape\label{rmk2.3}
We note for $\lambda>-1$,
$$D^{\alpha}_{0+}t^{\lambda}=\frac{\Gamma(\lambda+1)}{\Gamma(\lambda-\alpha+1)}t^{\lambda-\alpha}.$$
giving in particular $D^{\alpha}_{0+}t^{\alpha-m}=0$, $m=1,2,\cdots,N$, 
where $N$ is the smallest integer greater than or equal to $\alpha$.
\end{remark}
\begin{definition}
We first introduce the basic Banach spaces 
\begin{itemize}
\item $AC[0,1]$ : the space of absolute continuous functions on $[0,1]$ 
\item $AC^k[0,1]$: the space of real-valued functions $f$ which have continuous derivatives up to order $k-1$ on $[0,1]$ such that $f^{(k-1)}\in AC[0,1].$  
%\item $AC_{loc}(0,1)$ : the space consisting of functions that are absolutely continuous on every closed interval $[a,b] \subseteq (0,1)$.
\end{itemize}
In the book \cite{P}, the authors introduce a new Banach space
\begin{itemize}
\item $C^1_{\gamma}[a,b]= \{u\in C[a,b]: (t-a)^{\gamma} u'(t)\in C[0,1] \}$ with the norm $\|u\|_{C^1_{\gamma}}=\|u\|_\infty+\|u'\|_{C_{\gamma}}$ where $0<\gamma<1$,  $\|u\|_{\infty}=\max_{t\in[0,1]}|u(t)|$ and $\|u\|_{C_{\gamma}}=\max_{t\in[0,1]}|t^{\gamma}u(t)| $
\end{itemize}
Next, we define a new space $E_\alpha [0,1]$.
\begin{itemize}
\item $E_\alpha= \{u\in C[0,1] : t^{\alpha-1} D^{\alpha-1}_{0+}u(t) \in C[0,1] \}$ equipped with the norm $\|u\|_{E_\alpha}=\|u\|_{\infty}+\|u\|_1$ where $\|u\|_1=\max_{t\in[0,1]} |t^{\alpha-1} D^{\alpha-1}_{0+}u(t)|$.
\end{itemize}
\end{definition}
 Clearly, $\|\cdot \|_{E_\alpha}$ is a norm. Indeed, for any $u,v \in E_\alpha$ and $a\in \mathbb{R}^n$, we have by the linearity of $D^{\alpha-1}_{0+}$
$$\|au\|_{E_\alpha}=\|au\|_{\infty}+\|au\|_1=|a|(\|u\|_{\infty}+\|u\|_1)=|a|\|u\|_{E_\alpha}$$
and $$\|u+v\|_{E_\alpha}=\|u+v\|_{\infty}+\|u+v\|_1\le \|u\|_{E_\alpha}+\|v\|_{E_\alpha}.$$
In addition, for any $u\in E_\alpha ,$ if $0=\|u\|_{E_\alpha}=\|u\|_{\infty}+\|u\|_1,$ then $\|u\|_{\infty}=0$ and therefore $u=0.$
First of all, let us assume that $h$ is continuous. In \cite{E}, Bai and L\"{u} showed that there exists a solution in $C[0,1]$ using the fixed point theorem. Moreover, we can proved that all solutions for \ref{FDE} are of $C^1_{2-\alpha}[0,1]$ when $h\in C[0,1]$. See Appendix. Second, we assume that $h$ is in $C(0,1]$ and satisfies $(H)$. Then, we say that $u$ is a solution of \ref{FDE} if $u\in E_\alpha [0,1]\cap AC[0,1]$ and $u$ satisfies the equation \ref{FDE}. In our case, the solution $u\notin C^1_{2-\alpha}[0,1]$ but $u\in E_\alpha [0,1]$. Example \ref{ex3.2} shows that a given function $h_3$ corresponding to the solution $u_3$ satisfies our conditions $(H)$ and $u_3 \in  E_\alpha$ and $u_3\notin C^1_{2-\alpha}[0,1]$.
\begin{example}\label{ex3.2}
Consider three solutions $u_1(t)=t^{2.2}(1-t)$,$u_2(t)=t^{0.8}(1-t)$ and $u_3(t)=t^{0.2}(1-t)$ .
Using Remark \ref{rmk2.3}, we get 
$$D^{1.5}_{0+}u_1(t)=h_1(t):= \frac{\Gamma (3.2)}{\Gamma (2.7)}t^{0.7}-\frac{\Gamma (4.2)}{\Gamma (2.7)}t^{1.7}$$
$$D^{1.5}_{0+}u_2(t)=h_2(t):= \frac{\Gamma (1.8)}{\Gamma (0.3)}t^{-0.7}-\frac{\Gamma (2.8)}{\Gamma (1.3)}t^{0.3}$$
and $$D^{1.5}_{0+}u_3(t)=h_3(t):= \frac{\Gamma (1.2)}{\Gamma (-0.3)}t^{-1.3}-\frac{\Gamma (2.2)}{\Gamma (0.7)}t^{-0.3}$$
We also have $$(u_2)'(t)=0.8t^{-0.2}-1.8t^{0.8}$$
and $$(u_3)'(t)=0.2t^{-0.8}-1.2t^{0.2}$$
 Since $t^{0.5}(u_2)'(t)=0.8t^{0.3}-1.8t^{1.3}$, we can conclude that $u_2\in C^1_{0.5}[0,1]$ but $u_3\notin C^1_{0.5}[0,1]$. On the other hand, we have $t^{\alpha-1} D^{\alpha-1}_{0+}u_3(t) =t^{0.5} D^{0.5}_{0+}u_3(t) =t^{0.5}(\frac{\Gamma (1.2)}{\Gamma (0.7)}t^{-0.3}-\frac{\Gamma (2.2)}{\Gamma (1.7)}t^{0.7})\in C[0,1]$ and then $u_3\in E_\alpha$.
\end{example}
From this example, we can consider the proposition as follows:
\begin{Conjecture}\label{conjregu}
Let us consider the following equation
\begin{equation*}
\begin{cases}
 D^{\alpha}_{0+}u(t)+h(t)= 0,\quad t\in (0,1),\\
u(0)= 0 = u(1),
\end{cases}
\end{equation*}
where $D^{\alpha}_{0+}$ is the Riemann-Liouville fractional derivative of order $\alpha \in (1,2]$. If $h(t) \in C[0,1]$ or $L^1(0,1)$, the solution $u$ would be in $C^1_{2-\alpha}[0,1]$. Moreover, if $h(t) \in C(0,1]$ satisfying $(H)$, the solution $u$ would be in $E_\alpha [0,1]$.
\end{Conjecture}
In order to prove this claim, we introduce several properties of fractional derivative.
\begin{lemma}{\upshape(\cite{E}) }\label{lem2.4}
Assume that $u\in C(0,1)\cap L(0,1)$. For $\alpha>0$, $D^{\alpha}_{0+} u(t)=0$ has a unique solution
$$u(t)=c_1 t^{\alpha-1}+c_2 t^{\alpha-2}+\cdots +c_n t^{\alpha-n}, \quad c_i \in \mathbb{R}, i=1,2,\cdots,n $$
 where n is the smallest integer greater than or equal to $\alpha$. 
\end{lemma}
As $D^{\alpha}_{0+}I^{\alpha}_{0+} u(t)=u(t)$ for all $u\in C(0,1)\cap L(0,1)$. From Lemma \ref{lem2.4}, we deduce the following statement.
\begin{lemma}{\upshape(\cite{E},\cite{O}) }
Assume that $u\in C(0,1)\cap L(0,1)$ with a fractional derivative of order $\alpha>0$ that belongs to $C(0,1)\cap L(0,1)$. Then
$$I^{\alpha}_{0+}D^{\alpha}_{0+} u(t)=u(t)+c_1 t^{\alpha-1}+c_2 t^{\alpha-2}+\cdots +c_n t^{\alpha-n}, \quad c_i \in \mathbb{R}, i=1,2,\cdots,n .$$
Moreover, if $0<\alpha <1$ and $u(t) \in C[0,1]$, then $D^{\alpha}_{0+} u(t) \in C(0,1)\cap L(0,1)$ and $$I^{\alpha}_{0+}D^{\alpha}_{0+} u(t)=u(t).$$
\end{lemma}
\begin{lemma}{\upshape(\cite{K}) }\label{lm2.7}
Let the functions $\varphi\in AC([a,b];\mathbb{R})$ and $f:[c,d]\times [a,b]\rightarrow \mathbb{R}$ be such that the following relations hold and $\varphi([a,b])=[c,d]$:
\begin{align*}
f(\cdot, x )\in L([c,d];\mathbb{R}) \quad for ~ all~ x\in [a,b], \\
f(t,\cdot )\in AC([a,b];\mathbb{R}) \quad for ~ a.e.~ t\in [c,d], \\
\end{align*}
and $$f'_{[2]} \in L([c,d]\times [a,b];\mathbb{R}).$$
Put 
$$F(\lambda):=\int^{\varphi(\lambda)}_c f(t,\lambda) dt ~for ~\lambda \in [a,b].$$
Then, the following assertions are satisfied:
\begin{itemize}
\item[(a)] There exist sets $E_1\subseteq [c,d]$ and $E_2\subseteq [a,b]$ such that $meas E_1=d-c, meas E_2=b-a$, and $$F'(\lambda)=f(\varphi(\lambda),\lambda)\varphi '(\lambda)+\int^{\varphi (\lambda)}_c f'_{[2]}(t,\lambda)dt ~for ~a.e. ~\lambda \in \varphi^{-1}(E_1)\cap E_2.$$
\item[(b)] If the function $\varphi$ is monotone (not strictly, in general) then the function $F$ is absolutely continuous on the interval $[a,b]$.
\item[(c)] If the function $\varphi$ is strictly monotone then
$$F'(\lambda)=f(\varphi (\lambda),\lambda)\varphi '(\lambda)+\int^{\varphi (\lambda)}_c f'_{[2]}(t,\lambda)dt ~for ~a.e. \lambda \in [a,b].$$
\end{itemize}
\end{lemma}

%%%%%%%%%%%%%%%%%%%%%%%%%%%%%%%%%%%%%%%%%%%%%%%%%%%%%%%%%%%%%%%%%%%%%%
\section{Green's function}
When $h$ is continuous, it is well known about the Green's function related to problem $\eqref{eq:::1.1}$ by Lemma 2.3 in \cite{E} summarized by the following remark.
\begin{remark}\label{rmk3.1}(\cite{E})
Assume that $h \in C[0,1]$ and {$1<\alpha\leq2$}. Then
the unique solution of
\begin{equation}\label{eq3.1}
\begin{cases}
 D^{\alpha}_{0+}u(t)+h(t)= 0,\quad t\in (0,1),\\
u(0)= 0 = u(1)
\end{cases}
\end{equation}
 can be represented by
$$u(t)={ \int_{0}^{1} G(t,s)h(s)ds,}$$
where $G (t,s)$ is given in \eqref{eq:::1.2}.
\end{remark}
Let us consider the case if $h$ is singular at $t=0$ so that it is not integrable near $t=0$. To find the Green's function, Bai and L\"u \cite{E} take the Riemann-Liouville fractional integral $I^{\alpha}_{0+}$ on both sides of \eqref{eq3.1} for continuous case. 
However, when we consider, for example, $h(t)=t^{-1.5}$, we cannot derive $G(t,s)$ along with the idea of Bai and L\"u \cite{E}, 
since $I^\alpha_{0+} h(t)=\frac{1}{\Gamma (\alpha)}\int_0^t (t-s)^{\alpha -1}s^{-1.5} ds$ is not well-defined. Therefore, we need to try some other approach for singular case. Before we find the Green's function, we consider the definition of a solution for the equation \eqref{eq3.1} when $h$ is continuous and $h$ has singularity at 0. Now we give a lemma related to Green's function for singular case.
We note that Green's function for singular case derived in the following lemma has the same expression as continuous case.
%In this section, we \textcolor{red}{generalize the condition of $h$ for problem $(P)$.} 
\begin{theorem}\label{lemma3.2}
Assume $g$ satisfies $(H)$, then the following equation 
%Given $g \in L^{1}_{loc}((0,1),[0,\infty))$ satisfying for {$1<\alpha\leq2$}, \\
%$(H1) \int_0^1 s^{\alpha -1}g(s)ds<\infty$,\\
%$(H2)$ g is bounded on any compact subinterval in $(0,1],$\\
\begin{equation*}\tag*{($FDE_1)$}\label{P_1}
\begin{cases}
 D^{\alpha}_{0+}u(t)+g(t)= 0,\quad t\in (0,1),\\
u(0)= 0 = u(1),
\end{cases}
\end{equation*}
is equivalent to the functional integral equation: 
\begin{equation}\label{integ}
u(t)={ \int_{0}^{1} G (t,s)g(s)ds, }
\end{equation}
where $G (t,s)$ is given in \eqref{eq:::1.2}. Moreover, $u$ in \eqref{integ} are in $AC[0,1]\cap E_\alpha $ and $D^{\alpha-1}_{0+}u$ is absolutely continuous in any compact subinterval of $(0,1)$.
\end{theorem}
\begin{proof}
\noindent
Suppose that $u \in AC[0,1]\cap E_\alpha$ and $D^{\alpha-1}_{0+}u$ is absolutely continuous in any compact subinterval of $(0,1)$. Take $0<t<1$.
 By integrating both sides of equation \ref{P_1} from $t$ to $1$, we have
 $$ -\int_t^{1}D^{\alpha}_{0+}u(s)ds =\int_t^{1}g(s)ds.$$
 From definition of $D^{\alpha}_{0+}u(s)$=$(\frac{d}{ds})^{2}I^{2-\alpha}_{0+}u(s)$, we obtain
\begin{align} \label{(3.2)}
 -c_1+\frac{d}{dt}I^{2-\alpha}_{0+}u(t)=\int_t^{1}g(s)ds, 
 \end{align} 
  where $c_1=\frac{d}{dt}I^{2-\alpha}_{0+}u(t)|_{t=1}$. From the following inequalities, 
% \begin{align*}  \infty &>\int_0^{1}\tau^{\alpha-1} g(\tau)d\tau \geq \int^1_0 \tau g(\tau)d\tau =\int^1_0 \int_0^\tau ds g(\tau) d\tau \\ &=\int_0^{1}\int_s^{1}g(\tau)d\tau ds \geq \int_0^{t}\int_s^{1}g(\tau)d\tau ds. \end{align*}
\begin{align*}
&\int_0^{t}\int_s^{1}g(\tau)d\tau ds\le \int_0^{1}\int_s^{1}g(\tau)d\tau ds=
\int^1_0 \int_0^\tau ds g(\tau) d\tau\\ 
&= \int^1_0 \tau g(\tau)d\tau \le \int_0^{1}\tau^{\alpha-1} g(\tau)d\tau < \infty,
 \end{align*}
we see as a function of $t,$ $\int_t^{1}g(s)ds \in L(0,1)$ by condition $(H)$. 
Thus we can take the integration $\int_0^{t}$ on both sides of \eqref{(3.2)} and obtain
 \begin{eqnarray*}
 -c_1t-c_2+I^{2-\alpha}_{0+}u(t) &=& \int_0^{t}\int_s^{1}g(\tau)d\tau ds \\
 &=& \int_0^{t}\int_0^{\tau}g(\tau)dsd\tau + \int_t^{1}\int_0^{t}g(\tau)dsd\tau\\
 &=& \int_0^{t}\tau g(\tau)d\tau + \int_t^{1}t g(\tau)d\tau,
 \end{eqnarray*}
 where $c_2=I^{2-\alpha}_{0+}u(t)|_{t=0}$.  
 We have 
 \begin{align*}
 I^{2-\alpha}_{0+}u(t) =c_1t+c_2+ \int_0^{t}\tau g(\tau)d\tau + \int_t^{1}t g(\tau)d\tau.
 \end{align*}
 And by the definition of $D^{\alpha}_{0+}$, we obatin
\begin{align}\label{(3.3)}
D^{2-\alpha}_{0+}I^{2-\alpha}_{0+}u(t)&=\frac{d}{dt}\Big{(}\int_0^t u(s)ds\Big{)}\\&=\frac{c_1}{\Gamma(\alpha)}t^{\alpha-1}+\frac{c_2}{\Gamma(\alpha-1)}t^{\alpha-2}+D^{2-\alpha}_{0+}[\int_0^{t}\tau g(\tau)d\tau + \int_t^{1}t g(\tau)d\tau].
\end{align}
 We calculate the right part of \eqref{(3.3)}.  
 By the definition,
 \begin{align*}
 D^{2-\alpha}_{0+}[\int_0^{t}\tau g(\tau)d\tau + \int_t^{1}t g(\tau)d\tau]&=\frac{1}{\Gamma(\alpha-1)} \frac{d}{dt} \Big{[}\int_{0}^{t}{(t-s)^{\alpha-2}}\int_0^{s}\tau g(\tau)d\tau ds \\&+ \int_{0}^{t}{(t-s)^{\alpha-2}}\int_s^{1}s g(\tau)d\tau ds \Big{]}.
\end{align*}
Thus, we have
\begin{align}
 &\frac{1}{\Gamma(\alpha-1)} \frac{d}{dt} \Big{[}\int_{0}^{t}{(t-s)^{\alpha-2}}\int_0^{s}\tau g(\tau)d\tau ds +\int_{0}^{t}{(t-s)^{\alpha-2}}\int_s^{1}s g(\tau)d\tau ds \Big{]}\nonumber\\
 &=\frac{1}{\Gamma(\alpha-1)} \frac{d}{dt} \Big{[}\int_{0}^{t}\int_\tau^{t}{(t-s)^{\alpha-2}}\tau g(\tau)ds d\tau +\int_{0}^{t}\int_0^{\tau}{(t-s)^{\alpha-2}}s g(\tau)ds d\tau \nonumber\\
 &+ \int_{t}^{1}\int_0^{t}{(t-s)^{\alpha-2}}s g(\tau)ds d\tau \Big{]} \nonumber\\
 &=\frac{d}{dt}\frac{1}{\Gamma(\alpha-1)}\Big[ \int_{0}^{t}\frac{(t-\tau)^{\alpha-1}}{(\alpha-1)}\tau g(\tau)d\tau \nonumber\\&+ \int_{0}^{t}\Big(-\frac{(t-\tau)^{\alpha-1}}{(\alpha-1)}\tau - \frac{(t-\tau)^{\alpha}}{(\alpha-1)\alpha}+\frac{t^{\alpha}}{\alpha(\alpha-1)}\Big{)} g(\tau)d\tau +\int_t^1 \frac{t^{\alpha}}{\alpha(\alpha-1)} g(\tau)d\tau \nonumber\\
 &=\frac{d}{dt}\frac{1}{\Gamma(\alpha-1)}\Big[ \int_{0}^{t}\Big(\frac{t^{\alpha}}{\alpha(\alpha-1)} - \frac{(t-\tau)^{\alpha}}{(\alpha-1)\alpha}\Big{)} g(\tau)d\tau +
 \int_t^1 \frac{t^{\alpha}}{\alpha(\alpha-1)} g(\tau)d\tau \label{(3.4)}
  \end{align}
  We note that the integration $\int_{0}^{t}\Big(\frac{t^{\alpha}}{\alpha(\alpha-1)} - \frac{(t-\tau)^{\alpha}}{(\alpha-1)\alpha}\Big{)} g(\tau)d\tau$ is well-defined, indeed
  \begin{align}
  \label{pro}
  \int_{0}^{t}\Big(\frac{t^{\alpha}}{\alpha(\alpha-1)} - \frac{(t-\tau)^{\alpha}}{(\alpha-1)\alpha}\Big{)} g(\tau)d\tau &= \int_{0}^{t}\int^t_{t-\tau}\frac{s^{\alpha -1}}{\alpha -1}dsg(\tau)d\tau \\
  &\leq \int_{0}^{t} \frac{t^{\alpha-1}}{\alpha -1}\int^t_{t-\tau}dsg(\tau)d\tau \\
  &=\int_{0}^{t} \frac{t^{\alpha-1}}{\alpha -1} \tau g(\tau) d\tau <\infty .
  \end{align}
Hence, we get by Lemma \ref{lm2.7}
  \begin{eqnarray*}
  & &\frac{1}{\Gamma(\alpha-1)} \frac{d}{dt} \Big{[}\int_{0}^{t}{(t-s)^{\alpha-2}}\int_0^{s}\tau g(\tau)d\tau ds +\int_{0}^{t}{(t-s)^{\alpha-2}}\int_s^{1}s g(\tau)d\tau ds \Big{]}\\
 &=& \int_{0}^{t}\frac{t^{\alpha-1}-(t-\tau)^{\alpha-1}}{\Gamma(\alpha)}g(\tau)d\tau+\int_{t}^1\frac{ t^{\alpha-1}}{\Gamma(\alpha)}g(\tau)d\tau.
 \end{eqnarray*}
By the similar argument in \eqref{pro}, we can get the fact that
{\footnotesize
\begin{align}\label{eq33}
 & \int_{0}^{t}\frac{t^{\alpha-1}-(t-\tau)^{\alpha-1}}{\Gamma(\alpha)}g(\tau)d\tau =\int_{0}^{t}\frac{\alpha -1}{\Gamma(\alpha)}\int_{t-\tau}^t s^{\alpha -2}ds g(\tau)d\tau 
 \leq \int_{0}^{t}\frac{\alpha -1}{\Gamma(\alpha)} (t-\tau)^{\alpha -2}\tau g(\tau)d\tau \\
  &\leq \int_{0}^{\frac{t}{2}}\frac{\alpha -1}{\Gamma(\alpha)} (t-\tau)^{\alpha -2}\tau g(\tau)d\tau +\int_{\frac{t}{2}}^t \frac{\alpha -1}{\Gamma(\alpha)} (t-\tau)^{\alpha -2}\tau g(\tau)d\tau 
 < \infty .
 \end{align}
}  
The last inequality is valid. Indeed, 
$a(\tau)\triangleq (t-\tau)^{\alpha-2}$ is continuous at $\tau \in [0, \frac{t}{2}]$ and 
$\tau g(\tau)$ is integrable on $(0, \frac{t}{2}]$ by $(H)$. 
Moreover, $a(\tau)$ is integrable on $[\frac{t}{2},t)$ and 
$\tau g(\tau)$ is bounded on $[\frac{t}{2},t].$
Hence, we can rewrite the equation \eqref{(3.3)} as follows
\begin{eqnarray*}
&&u(t)-\frac{c_1}{\Gamma(\alpha)}t^{\alpha-1}-\frac{c_2}{\Gamma(\alpha-1)}t^{\alpha-2}\\
&=&\int_{0}^{t}\frac{t^{\alpha-1}-(t-\tau)^{\alpha-1}}{\Gamma(\alpha)}g(\tau)d\tau+\int_{t}^1\frac{ t^{\alpha-1}}{\Gamma(\alpha)}g(\tau)d\tau.
\end{eqnarray*}
From $u(0)=u(1)=0$, we have the following condition
 \begin{eqnarray*}
& &  \frac{c_2}{\Gamma(\alpha-1)}=0,\\
&-& \frac{c_1}{\Gamma(\alpha)}= \int_{0}^{1}\frac{1-(1-\tau)^{\alpha-1}}{\Gamma(\alpha)}g(\tau)d\tau.
 \end{eqnarray*}
 The second equality is valid using \eqref{eq33} with $t=1$.
 %\begin{align}\label{seeq}
% \int_{0}^{1}\frac{1-(1-\tau)^{\alpha-1}}{\Gamma(\alpha)}g(\tau)d\tau 
% &=\frac{1}{\Gamma(\alpha)}\int_{0}^{1}\int_{1-\tau}^1 (\alpha -1)s^{\alpha -2}dsg(\tau)d\tau \\
 %&\leq \frac{1}{\Gamma(\alpha)}\int_{0}^{1}\int_{1-\tau}^1 (\alpha -1)(1-\tau)^{\alpha -2}dsg(\tau)d\tau \\
 %&=\frac{\alpha -1}{\Gamma(\alpha)} \int_0^1 (1-1+\tau)(1-\tau)^{\alpha -2}g(\tau)d\tau
 %\\ &=\frac{\alpha -1}{\Gamma(\alpha)} \int_0^1 (1-\tau)^{\alpha -2}\tau g(\tau)d\tau <\infty,
 %\end{align} 
 %where $(1-\tau)^{\alpha -2}\tau g(\tau)$ is in $L(0,1),$ by condition $(H).$  
Plugging boundary conditions in the equations, we have
 \begin{align*}
 u(t)&+\int_{0}^{1}\frac{t^{\alpha-1}-(t(1-\tau))^{\alpha-1}}{\Gamma(\alpha)}g(\tau)d\tau \\
 &=\int_{0}^{t}\frac{t^{\alpha-1}-(t-\tau)^{\alpha-1}}{\Gamma(\alpha)}g(\tau)d\tau+\int_{t}^1\frac{ t^{\alpha-1}}{\Gamma(\alpha)}g(\tau)d\tau .
 \end{align*} Therefore, we have 
  \begin{align*}
 u(t)&=\int_{0}^{t} \frac{(t(1-\tau))^{\alpha-1}-(t-\tau)^{\alpha-1}}{\Gamma(\alpha)}g(\tau)d\tau\\&+\int_t^1 \frac{(t(1-\tau))^{\alpha-1}}{\Gamma(\alpha)}g(\tau)d\tau .
 \end{align*}
 and then we can conclude that the solution of \ref{P_1} satisfies the function integral equation \eqref{integ}.
  Now, let us take $u$ as follows:
  $$
u(t)={ \int_{0}^{1} G (t,s)g(s)ds, }
$$
where $g$ satisfies $(H1), (H2).$
We want to show that $u \in AC[0,1]\cap E_\alpha[0,1]$ and 
$$D^\alpha_{0+}(u)(t)+g(t)=0, ~t\in (0,1).$$
First, we will prove $u\in AC[0,1].$ Since $\lim_{t\rightarrow 0} (u)(t)=\lim_{t\rightarrow 1} (u)(t)=0$, $u(t)$ is continuous. Next, we have
\begin{align*}
  u'(t)&=\int_{0}^{t} \frac{t^{\alpha-2}(1-\tau)^{\alpha-1}-(t-\tau)^{\alpha-2}}{\Gamma(\alpha -1)}g(\tau)d\tau+\int_t^1 \frac{t^{\alpha-2}(1-\tau)^{\alpha-1}}{\Gamma(\alpha)}g(\tau)d\tau \\
  &=t^{\alpha -2}\Big{[}\int_{0}^{t} \frac{(1-\tau)^{\alpha-1}-(1-\frac{\tau}{t})^{\alpha-2}}{\Gamma(\alpha -1)}g(\tau)d\tau+\int_t^1 \frac{(1-\tau)^{\alpha-1}}{\Gamma(\alpha)}g(\tau)d\tau\Big].
  \end{align*}
In order to prove $\int_0^1 |u'(t)| dt <\infty$, we need to check the integrability of some functions. For fixed $t\in (0,1],$ let us consider a function ${\eta}_1$ defined by 
$$\eta _1 (s)= \frac{(t-s)^{\alpha-2}-t^{\alpha-2}(1-s)^{\alpha-1}}{s^{\alpha-1}},  $$
for $s\in (0, \frac{t}{2}].$
The numerator of ${\eta}_1$ satisfies the following inequality
%First, we consider ${(Su)}^{\prime}$. By boundedness of $u$ and continuity of $f$, we can obtain
\begin{eqnarray*}
& &(t-s)^{\alpha-2}-t^{\alpha-2}(1-s)^{\alpha-1}\\
&= & (t-s)^{\alpha-2}-t^{\alpha-2}(1-s)^{\alpha-2} (1-s) \\
&= & (t-s)^{\alpha-2}-t^{\alpha-2}(1-s)^{\alpha-2}+s t^{\alpha-2} (1-s)^{\alpha-2}\\
&= & \int_{t-s}^{t(1-s)} (2-\alpha) {\tau}^{\alpha -3} d\tau + s t^{\alpha-2} (1-s)^{\alpha-2}.
\end{eqnarray*}
Since $-2<\alpha -3 \le -1$ and $ 0<t-s\le t(1-s),$ for  $\tau \in [t-s, t(1-s)],$ we see
$\tau^{\alpha -3} \le {(t-s)^{\alpha -3}}.$ Thus
$$\int_{t-s}^{t(1-s)} (2-\alpha) {\tau}^{\alpha -3} d\tau \le (2-\alpha){(t-s)^{\alpha -3}} s(1-t)$$
and 
\begin{eqnarray*}
{\rm the \ numerator} \ &\le & (2-\alpha){(t-s)^{\alpha -3}} s(1-t)+ s t^{\alpha-2} (1-s)^{\alpha-2}\\
&=& s[(2-\alpha){(t-s)^{\alpha -3}} (1-t)+  t^{\alpha-2} (1-s)^{\alpha-2}].
\end{eqnarray*}
As a function of $s$ on $[0, \frac{t}{2}],$ the inside of the above bracket is continuous, so bounded by say, $A_t$ for all $s\in [0, \frac{t}{2}].$ Therefore
$$0\le {\eta_1}(s) \le {\frac{A_t s }{s^{\alpha -1 } }}= A_t s^{2-\alpha},$$
for all $s\in [0, \frac{t}{2}].$ 
Since $\alpha \in (1,2)$, this implies that ${\eta_1}(s) \to 0$ as $s\to 0$ and
 ${\eta_1}$ is continuous on $[0, \frac{t}{2}].$
Define a function $\eta_2$ given as 
$$\eta_2 = (t-s)^{\alpha-2}-t^{\alpha-2}(1-s)^{\alpha-1}.$$
Then $\eta_2 (s)$ is integrable on $({\frac{t }{2 }}, t).$ First, we check that $u'(t)$ is well-defined. 
\begin{eqnarray*}
&&|u'(t)|\\ 
&\le&  \int_{0}^{t} \frac{(t-s)^{\alpha-2}-t^{\alpha-2}(1-s)^{\alpha-1}}{\Gamma(\alpha-1)}g(s)ds +\int_{t}^{1}\frac{t^{\alpha-2}(1-s)^{\alpha-1}}{\Gamma(\alpha-1)}g(s)ds \\ 
&\leq & {\frac {1}{\Gamma(\alpha-1)}} \Big [\int_{0}^{t} ((t-s)^{\alpha-2}-t^{\alpha-2}(1-s)^{\alpha-1})g(s)ds+\int_{t}^{1}t^{\alpha-2}(1-s)^{\alpha-1}g(s)ds\Big]\\
&=& {\frac {1}{\Gamma(\alpha-1)}} \Big [ \int_0^{\frac{t }{ 2}} \frac{(t-s)^{\alpha-2}-t^{\alpha-2}(1-s)^{\alpha-1}}{s^{\alpha-1}} s^{\alpha-1}g(s)ds\\
& & \ \ \ \ \ + \int_{\frac{t }{ 2}}^t ((t-s)^{\alpha-2}-t^{\alpha-2}(1-s)^{\alpha-1})g(s)ds + \int_{t}^{1}t^{\alpha-2}(1-s)^{\alpha-1}g(s)ds\Big]\\
&=& {\frac {1}{\Gamma(\alpha-1)}} \Big [ \int_0^{\frac{t }{ 2}} {\eta_1}(s) s^{\alpha-1}g(s)ds
+ \int_{\frac{t }{ 2}}^t {\eta_2 }(s) g(s)ds+ \int_{t}^{1}t^{\alpha-2}(1-s)^{\alpha-1}g(s)ds\Big].
\end{eqnarray*}
First integration in the above bracket is well-defined, since $\eta_1$ is continous and 
$s^{\alpha-1}g(s)$ is integrable by $(H).$ The second is also well-defined, since $\eta_2$ and $h$ are integrable. The third is obviously well-defined, since $g$ is integrable on $[t,1].$

Now, we show $u' \in L^1 (0,1)$. Indeed,
\begin{eqnarray*}
& & \int_{0}^{1}|u'(t)|dt \leq  \Big [\int_{0}^{1}\int_{0}^{t}\frac{(t-s)^{\alpha-2}-t^{\alpha-2}(1-s)^{\alpha-1}}{\Gamma(\alpha-1)}g(s)dsdt\\
& & \ \ \ \ \ +\int_{0}^{1}\int_{t}^{1}\frac{t^{\alpha-2}(1-s)^{\alpha-1}}{\Gamma(\alpha)}g(s)_2dsdt\Big]\\
&\leq & \Big [\int_{0}^{1}\int_{s}^{1}\frac{(t-s)^{\alpha-2}-t^{\alpha-2}(1-s)^{\alpha-1}}{\Gamma(\alpha-1)}g(s)dtds\\
& & \ \ \ \ \ +\int_{0}^{1}\int_{0}^{s}\frac{t^{\alpha-2}(1-s)^{\alpha-1}}{\Gamma(\alpha-1)}g(s)dtds\Big]\\
&\leq & \int_{0}^{1}\frac{s^{\alpha-1}(1-s)^{\alpha-1}}{\Gamma(\alpha)}g(s)ds+\int_{0}^{1}\frac{s^{\alpha-1}(1-s)^{\alpha-1}}{\Gamma(\alpha)}g(s)ds\\
&=&2\int_{0}^{1}\frac{s^{\alpha-1}(1-s)^{\alpha-1}}{\Gamma(\alpha)}g(s)ds <\infty,
\end{eqnarray*}
by condition $(H)$. From this inequality, we can conclude $u(t)\in AC[0,1]$. By the definition, we have 
  \begin{align*}
  D^{\alpha-1}_{0+} u(t)&=\frac{1}{\Gamma(2-\alpha)} (\frac{d}{dt})\int_0^t (t-s)^{1-\alpha}u(s)ds\\
  &=\frac{1}{\Gamma(2-\alpha)} (\frac{d}{dt})[\int_0^t\int_0^s(t-s)^{1-\alpha} \frac{(s(1-\tau))^{\alpha-1}-(s-\tau)^{\alpha-1}}{\Gamma(\alpha)}g(\tau)d\tau ds \\
  &+\int_0^t\int_s^1 (t-s)^{1-\alpha} \frac{(s(1-\tau))^{\alpha-1}}{\Gamma(\alpha)}g(\tau)d\tau ds]\\
  & =\frac{1}{\Gamma(2-\alpha)} (\frac{d}{dt})[\int_0^t\int_\tau^t (t-s)^{1-\alpha} \frac{(s(1-\tau))^{\alpha-1}-(s-\tau)^{\alpha-1}}{\Gamma(\alpha)}g(\tau)ds d\tau \\
&  +\int_0^t\int_0^\tau (t-s)^{1-\alpha} \frac{(s(1-\tau))^{\alpha-1}}{\Gamma(\alpha)}g(\tau)ds d\tau \\ 
&  +\int_t^1\int_0^t (t-s)^{1-\alpha} \frac{(s(1-\tau))^{\alpha-1}}{\Gamma(\alpha)}g(\tau)ds d\tau ]\\
 &=\frac{1}{\Gamma(2-\alpha)\Gamma(\alpha)} (\frac{d}{dt})[\int_0^t\int_0^1 (1-z)^{1-\alpha}t (z(1-\tau))^{\alpha-1}g(\tau)dz d\tau \\&-\int_0^t\int_0^1 (1-z)^{1-\alpha}z^{\alpha-1}(t-\tau)g(\tau)dzd\tau \\
& + \int_t^1\int_0^1 (1-z)^{1-\alpha}t (z(1-\tau))^{\alpha-1}g(\tau)dz d\tau].
  \end{align*}
  And by the simple calculation in \cite{E}, we get 
   \begin{align*}
  D^{\alpha-1}_{0+} u(t)&=\frac{1}{\Gamma(2-\alpha)\Gamma(\alpha)} (\frac{d}{dt})[\int_0^t\int_0^1 (1-z)^{1-\alpha}t (z(1-\tau))^{\alpha-1}g(\tau)dz d\tau \\&-\int_0^t\int_0^1 (1-z)^{1-\alpha}z^{\alpha-1}(t-\tau)g(\tau)dzd\tau \\
& + \int_t^1\int_0^1 (1-z)^{1-\alpha}t (z(1-\tau))^{\alpha-1}g(\tau)dz d\tau]\\
&=\frac{\Gamma(2-\alpha)\Gamma(\alpha)}{\Gamma(2-\alpha)\Gamma(\alpha)} (\frac{d}{dt})\Big{[}\int_0^t [t(1-\tau)^{\alpha -1}-(t-\tau)]g(\tau) d\tau  + \int_t^1t (1-\tau)^{\alpha-1}g(\tau)d\tau \Big{]}\\
=& \frac{\Gamma(2-\alpha)\Gamma(\alpha)}{\Gamma(2-\alpha)\Gamma(\alpha)}[\int_0^t ((1-\tau)^{\alpha-1}-1)g(\tau)d\tau +\int_t^1 (1-\tau)^{\alpha-1}g(\tau)d\tau]\\
&=\int_0^t ((1-\tau)^{\alpha-1}-1)g(\tau)d\tau+\int_t^1 (1-\tau)^{\alpha-1}g(\tau)d\tau.
    \end{align*}
        By the inequality \eqref{eq33}, we get the fact that $((1-\tau)^{\alpha-1}-1)g(\tau) \in L^1(0,t)$ for $t\in (0,1)$ and $ (1-\tau)^{\alpha-1}g(\tau) \in L^1(t,1)$ for $t\in (0,1)$.  Therefore, we have 
         \begin{align}\label{Dalpham1}
  D^{\alpha-1}_{0+} u(t)=\int_0^t ((1-\tau)^{\alpha-1}-1)g(\tau)d\tau+\int_t^1 (1-\tau)^{\alpha-1}g(\tau)d\tau.
    \end{align}
        From \eqref{Dalpham1} and Lemma \ref{lm2.7}, we get
        \begin{align*}
     D^{\alpha}_{0+} u(t)&=\frac{d}{dt}D^{\alpha-1}_{0+} u(t)=\frac{d}{dt}(\int_0^t ((1-\tau)^{\alpha-1}-1)g(\tau)d\tau +\int_t^1 (1-\tau)^{\alpha-1}g(\tau)d\tau\\
     &=((1-t)^{\alpha-1}-1)g(t)-(1-t)^{\alpha-1}g(t)=-g(t)
            \end{align*}
            and therefore
           \begin{align}\label{Dalpha}
           D^\alpha_{0+} u(t)+g(t)=0, ~~ t\in(0,1).
           \end{align}
     Now we want to show that $u\in E_\alpha$, i.e., $t^{\alpha-1}D^{\alpha-1}_{0+} u(t) \in C[0,1]$ . 
     From \eqref{Dalpham1} and \eqref{Dalpha}, we have
    \begin{align*}
    &\int_0^1 |[t^{\alpha-1}D^{\alpha-1}_{0+} u(t)]'|dt= \int_0^1 |[t^{\alpha-2}D^{\alpha-1}_{0+} u(t)+t^{\alpha-1}D^{\alpha}_{0+} u(t)]|dt\\
   &=\int_0^1 \int_0^t t^{\alpha-2}(1-(1-\tau)^{\alpha-1})g(\tau)d\tau dt+\int_0^1\int_t^1 t^{\alpha-2}(1-\tau)^{\alpha-1}g(\tau)d\tau dt\\
   &+\int_0^1 -t^{\alpha-1}g(t)dt\\
   &\le \frac{1}{\alpha-1}[\int_0^t (1-(1-\tau)^{\alpha-1})g(\tau)d\tau+\int_t^1 (1-\tau)^{\alpha-1}g(\tau)d\tau ]\\
   &+\int_0^1 -t^{\alpha-1}g(t)dt<\infty.
       \end{align*}
       We can conclude that $t^{\alpha-1}D^{\alpha-1}_{0+} u(t)$ is absolutely continuous and therefore  $u(t)\in E_\alpha$. 
\end{proof}
 %Since $s^{\alpha-1}g(s)$ is bounded on any compact subset in $(0,1]$,  
\begin{example}
We examine the graph of numerical solution for the following problem
\begin{align*}%\label{(4.2)}
\begin{cases}
 D^{\alpha}_{0+}u(t)+h(t)= 0,\quad \alpha=1.6\\
u(0)= 0 = u(1),
\end{cases}
\end{align*}
where $f(t,u)=h(t).$ We choose $h(t)=t^{0.6}, t^{0}, t^{-0.6}$, and $t^{-1.2}$ and graph $u(t)$ using Green's function $G(t,s)$ defined by \eqref{eq:::1.2}. Figure \ref{f1} shows that the solution $u(t)$ is continuous when $h(t)=t^{-1.2}$ which has a singularity at 0. However, the slope of $u$ goes to infinity as $t$ goes to zero.
%when $\alpha=1.6$. 
\begin{figure}[h]
\advance\leftskip 2cm
\includegraphics[width=9cm]{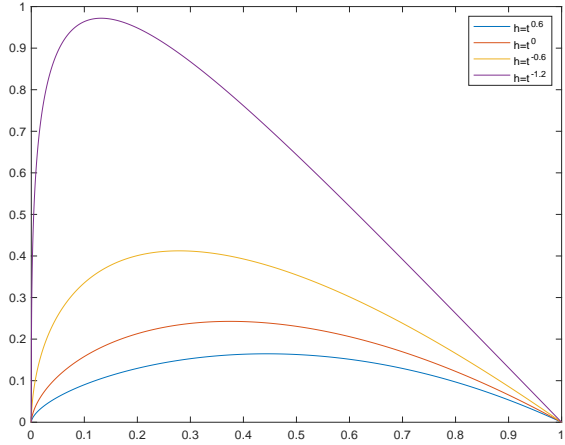}
\caption{Graph of $u(t)$ when $h(t)=t^{0.6}, t^{0}, t^{-0.6}$, and $t^{-1.2}$ and $\alpha=1.6 .$}\label{f1}
\end{figure}
\end{example}
\newpage
\appendix
\section{$C^1_{2-\alpha}[0,1]$-Regularity}
In this section, we consider the regularity of our problem when $h$ is continuous or integrable. 
Let us assume that $h$ is continuous. In \cite{E}, the researchers proved that for given $h\in C[0,1]$ and $1<\alpha \le 2,$ the unique solution of 
\begin{equation*}
\begin{cases}
 D^{\alpha}_{0+}u(t)+h(t)= 0,\quad t\in (0,1),\\
u(0)= 0 = u(1),
\end{cases}
\end{equation*}
 is
\begin{align}\label{solint}
u(t)=\int_0^1 G(t,s)h(s)ds,
\end{align}
where $G(t,s)$ is given in \eqref{eq:::1.2}. We want to prove that $u$ in \eqref{solint} is in $C^1_{2-\alpha}[0,1]$, i.e. $t^{2-\alpha}u'(t) \in C[0,1].$
By Lemma \ref{lm2.7}, we get
\begin{align*}
u'(t)&=\frac{d}{dt}\Big{(}\int_0^1 G(t,s)h(s)ds\Big{)}\\
&= \int_0^t \frac{(\alpha -1)}{\Gamma(\alpha)}(t^{\alpha-2}(1-s)^{\alpha-1}-(t-s)^{\alpha-2})h(s)ds\\
&~~+\int_t^1 \frac{(\alpha -1)}{\Gamma(\alpha)} t^{\alpha-2}(1-s)^{\alpha-1}h(s)ds.
\end{align*}
So, we have
\begin{align*}
t^{2-\alpha}u'(t)&= \int_0^t \frac{(\alpha -1)}{\Gamma(\alpha)}((1-s)^{\alpha-1}-(1-\frac{s}{t})^{\alpha-2})h(s)ds\\
&~~+\int_t^1 \frac{(\alpha -1)}{\Gamma(\alpha)} (1-s)^{\alpha-1}h(s)ds.
\end{align*}
Since $h$ is continuous, we get
\begin{align*}
t^{2-\alpha}u'(t)&= \int_0^t \frac{(\alpha -1)}{\Gamma(\alpha)}((1-s)^{\alpha-1}-(1-\frac{s}{t})^{\alpha-2})h(s)ds\\
&~~+\int_t^1 \frac{(\alpha -1)}{\Gamma(\alpha)} (1-s)^{\alpha-1}h(s)ds\\
&\le M [\int_0^t \frac{(\alpha -1)}{\Gamma(\alpha)}((1-s)^{\alpha-1}-(1-\frac{s}{t})^{\alpha-2})ds\\
&~~+\int_t^1 \frac{(\alpha -1)}{\Gamma(\alpha)} (1-s)^{\alpha-1}ds]<\infty .
\end{align*}

Second, we suppose that $h\in L^1(0,1)\cap C(0,1).$ With the similar technique to the proof of Lemma 2.3 in \cite{E}, we can conclude that 
\begin{align}
u(t)=\int_0^1 G(t,s)h(s)ds.
\end{align}
Hence, we have
\begin{align*}
t^{2-\alpha}u'(t)&= \int_0^t \frac{(\alpha -1)}{\Gamma(\alpha)}((1-s)^{\alpha-1}-(1-\frac{s}{t})^{\alpha-2})h(s)ds\\
&~~+\int_t^1 \frac{(\alpha -1)}{\Gamma(\alpha)} (1-s)^{\alpha-1}h(s)ds\\
&= \int_0^{t/2} \frac{(\alpha -1)}{\Gamma(\alpha)}((1-s)^{\alpha-1}-(1-\frac{s}{t})^{\alpha-2})h(s)ds\\
&+ \int_{t/2}^t \frac{(\alpha -1)}{\Gamma(\alpha)}((1-s)^{\alpha-1}-(1-\frac{s}{t})^{\alpha-2})h(s)ds\\
&~~+\int_t^1 \frac{(\alpha -1)}{\Gamma(\alpha)} (1-s)^{\alpha-1}h(s)ds.
\end{align*}
By simple calculation, we get the following fact 
\begin{itemize}
 \item[1.] $(1-s)^{\alpha-1}-(1-\frac{s}{t})^{\alpha-2} \in C[0,t/2]$ and $h(s)\in L^1(0,t/2)$,  
\item [2.] $(1-s)^{\alpha-1}-(1-\frac{s}{t})^{\alpha-2} \in L^1(t/2,t)$ and $h(s)\in C[t/2, t]$,
\item [3.]  $(1-s)^{\alpha-1}\in C[t,1]$ and $h(s)\in L^1(t,1).$
\end{itemize}
Therefore, we can prove that $t^{2-\alpha}u'(t) \in C[0,1].$
As a results, we can prove that the statement in Conjecture \ref{conjregu} as follows:
\begin{lemma}
Let us consider the following equation
\begin{equation*}
\begin{cases}
 D^{\alpha}_{0+}u(t)+h(t)= 0,\quad t\in (0,1),\\
u(0)= 0 = u(1),
\end{cases}
\end{equation*}
where $D^{\alpha}_{0+}$ is the Riemann-Liouville fractional derivative of order $\alpha \in (1,2]$. If $h(t) \in C[0,1]$ or $L^1(0,1)\cap C(0,1)$, the solution $u$ would be in $C^1_{2-\alpha}[0,1]$. Moreover, if $h(t) \in C(0,1]$ satisfying $(H)$, the solution $u$ would be in $E_\alpha [0,1]$.
\end{lemma}
%%%%%%%%%%%%%%%%%%%%%%%%%%%%%%%%%%%%%%%%%%%%%%%%%%%%%%%%%%%%%%%
%%%%%%%%%%%%%%%%%%%%%%%%%%%%%%%%%%%%%%%%%%%%%%%%%%%%%%%%%%%%%%%%%%%%%%
\section*{Acknowledgment}
The authors express their gratitude to anonymous referees for their helpful suggestions which improved final version of this paper.

\section*{Funding}
This work was supported by the National Research Foundation of Korea, Grant funded by the Korea Government (MEST) (NRF2016R1D1A1B04931741).

\section*{Availability of data and materials}
Data sharing not applicable to this article as no datasets were generated or analyzed during the current study.

\section*{Competing interests}
The authors declare that there is no competing interests for this paper.

\section*{Author's contributions}
All authors have equally contributed in obtaining new results in this article and also
read and approved the final manuscript.

\section*{Author's information}
  Jinsil Lee, Department of Mathematics, University of Georgia,  Athens, GA 30606, USA.
   E-mail: jl74942@uga.edu\\
 Yong-Hoon Lee, Department of Mathematics, Pusan National University, Busan 46241, Republic of Korea.
    E-mail: yhlee@pusan.ac.kr

\section*{Publisher's Note}
Springer Nature remains neutral with regard to jurisdictional claims in published maps and institutional affiliations.

%\iffalse
 %%%%%%%%%%%%%%%%%%%%%%%%%%%%%%%%%%%%%%%%%%%%%%%%%%%%%%%%%%%%%%%%%%%%%%%
\bibliographystyle{plain}
%\newpage

\end{document}